\documentclass[12pt,reqno]{amsart}

\usepackage[T1]{fontenc}
\usepackage{lmodern}
\usepackage{tikz-cd}

\usepackage{amsmath,amssymb,amsthm,mathrsfs,mathtools}
\usepackage{marvosym}

\usepackage{xcolor}
\usepackage{tikz}
\usetikzlibrary{arrows.meta,positioning}

\usepackage[a4paper,margin=1.25in]{geometry}
\usepackage[
    colorlinks,
    allcolors=black,
    pdfauthor={},
    pdftitle={}
]{hyperref}
\usepackage[nameinlink]{cleveref}

\theoremstyle{plain}
\newtheorem{theorem}{Theorem}[section]
\newtheorem{proposition}[theorem]{Proposition}
\newtheorem{corollary}[theorem]{Corollary}
\newtheorem{lemma}[theorem]{Lemma}

\theoremstyle{definition}

\newtheorem{question}[theorem]{Question}

\DeclareSymbolFont{bbold}{U}{bbold}{m}{n}
\DeclareSymbolFontAlphabet{\mathbbold}{bbold}

\DeclareMathOperator{\cone}{cone}
\DeclareMathOperator{\rank}{rank}
\DeclareMathOperator{\Span}{span}

\DeclareMathOperator{\ran}{ran}

\begin{document}

 \title[Nonnegative rank]{On the nonnegative rank of positive operators}

\author{Roman Drnov\v sek}
\address{Faculty of Mathematics and Physics, University of Ljubljana,
  Jadranska 19, 1000 Ljubljana, Slovenia \ \ and \ \ \
  Institute of Mathematics, Physics, and Mechanics, Jadranska 19, 1000 Ljubljana, Slovenia}
\email{roman.drnovsek@fmf.uni-lj.si}

\author{Marko Kandi\'c$^\ast$}
\thanks{$^\ast$ Corresponding author \Letter\ \href{mailto:marko.kandic@fmf.uni-lj.si}{marko.kandic@fmf.uni-lj.si}}
\address{Faculty of Mathematics and Physics, University of Ljubljana,
  Jadranska 19, 1000 Ljubljana, Slovenia \ \ and \ \ \
  Institute of Mathematics, Physics, and Mechanics, Jadranska 19, 1000 Ljubljana, Slovenia}
\email{marko.kandic@fmf.uni-lj.si}


\date{\today}

\begin{abstract}
In this paper we introduce the concept of a nonnegative rank of a positive operator $T\colon X\to Y$ between ordered vector spaces. In the case of nonnegative matrices, our definition agrees with the standard definition of a nonnegative rank. Under some natural and mild assumptions on the cone $Y_+$, we prove that the nonnegative rank and the rank agree whenever the rank is at most two. This can be considered as the infinite-dimensional version of \cite[Theorem 4.1]{CR93}. We also provide an example of a positive rank-three operator on the Banach lattice $C[0,1]$ with an infinite nonnegative rank.
\end{abstract}

\maketitle

\section{Introduction and preliminaries}

Let $T$ be a nonnegative $m\times n$ matrix. Gregory and Pullman introduced in \cite{GP83} the \emph{nonnegative rank} $\rank^+(T)$ of $T$ as the smallest nonnegative integer $k$ such that $T$ can be expressed as a sum of $k$ rank-one nonnegative matrices. Equivalently, $k$ is the nonnegative rank of $T$ if and only if there exist nonnegative matrices $L \in \mathbb{R}^{m \times k}$ and $R \in \mathbb{R}^{k \times n}$ such that $T = LR$, where $k$ is minimal with this property (see, for example, \cite[Corollary 2.2]{CR93}). 

If $\rank(T)\leq 2$, then \cite[Theorem 4.1]{CR93} implies that $\rank^+(T)=\rank(T)$. In contrast, for matrices of rank three, the nonnegative rank can be strictly larger. Indeed, Robbins (see \cite{CR93}) provided an example of a nonnegative $4\times 4$ matrix of 
rank three
$$
T=
\begin{pmatrix}
1 & 1 & 0 & 0 \\
1 & 0 & 1 & 0 \\
0 & 1 & 0 & 1 \\
0 & 0 & 1 & 1
\end{pmatrix}
$$
for which $\rank^+(T)=4$. This naturally raises the question of how large the nonnegative rank of a matrix of rank three can be. A special case of \cite[Proposition 3.1]{BL09} shows that it can, in fact, be arbitrarily large. More precisely, for any $k \geq 3$ there exist $n \in \mathbb{N}$ and an $n\times n$ nonnegative matrix $T$ such that $\rank^+(T)=k$ and $\rank(T)=3$. Moreover, an inspection of the proof reveals that one may take any $n \geq \binom{k}{\lfloor k/2 \rfloor}$.

It should be clear that the nonnegative rank of a nonnegative $m\times n$ matrix is at most $\min\{m,n\}$. 
In certain special cases, however, this upper bound can be improved. In particular, Shitov proved in \cite[Theorem 3.2]{Shi14} that every nonnegative $m\times n$ matrix of rank three has nonnegative rank at most $\lceil 6\min\{m,n\}/7\rceil$. More recently, Shitov showed in \cite[Corollary 3 and Theorem 5]{Shi26} that if a nonnegative matrix of rank three has at most $n\geq 3$ rows, then its nonnegative rank is bounded above by $147n^{2/3}$.

The natural notion of the nonnegative rank of a finite-rank positive operator between ordered vector spaces was introduced by the authors in \cite{DK26}. They investigated nonnegative matrices and positive operators on $\ell^p$ and $L^p[0,1]$ ($1 \leq p < \infty$) that can be represented as commutators of square-zero nonnegative matrices and positive operators, respectively. In the present paper, we initiate a systematic study of the concept of the nonnegative rank. In \Cref{sec:nonnegative rank intro} we recall the general definition of the nonnegative rank of a positive finite rank operator between ordered vector spaces and derive many natural generalizations of results from \cite{CR93} that hold in the finite-dimensional case. Under some natural and mild assumptions on the cone $Y_+$, we prove in \Cref{sec:rank and nonnegative rank agree} that the nonnegative rank and the rank agree whenever the rank is at most two. This can be considered as the infinite-dimensional version of \cite[Theorem 4.1]{CR93}. In \Cref{sec:example} we provide an example of a positive rank three operator $\mathscr S$ with infinite nonnegative rank on the Banach lattices $C[0,2\pi]$ and $L^p[0,2\pi]$ $(1\leq p\leq \infty)$. The operator is the integral operator defined by 
$$(\mathscr S f)(s)=\int_0^{2\pi} f(t)(1+\cos(s-t))\,dt . $$
In \Cref{sec:S-range cone} we study the range cone $\ran \mathscr S \cap C[0,2\pi]_+$ and the cone  $\mathscr S (C[0,2\pi]_+)$. We prove that the cone $\mathscr S (C[0,2\pi]_+)$ is norm dense in the range cone $\ran \mathscr S \cap C[0,2\pi]_+$ which is order and topologically isomorphic to the ice cream cone 
$$\{(a,b,c)\in \mathbb R^3:\; c\geq \sqrt{a^2+b^2}\}.$$
 
In the remaining part of this section we establish some notation and terminology needed throughout the text. Let $(\Omega, \Sigma, \mu)$ be a measure space, where $\Omega$ is a set, $\Sigma$
is a $\sigma$-algebra of subsets of $\Omega$ and $\mu$ a $\sigma$-finite positive measure on $\Sigma$.
By $L^0(\Omega, \Sigma, \mu)$ we denote the vector lattice of all equivalence classes of measurable real functions on the set $\Omega$. 
By \cite[Theorem 1.80]{AA02}, this is a Dedekind complete vector lattice with the countable sup property.
Therefore, ideals of $L^0(\Omega, \Sigma, \mu)$ are also Dedekind complete vector lattices. In particular, the classical spaces $L^p(\Omega, \Sigma, \mu)$ ($1 \le p \le \infty$) are all Dedekind complete. 

Let $X$ be an order ideal in $L^0(\Omega, \Sigma, \mu)$. 
An operator $T$ on $X$ is said to be an \emph{integral operator} if there exists a measurable real function $k_T$ on 
$\Omega \times \Omega$ (called the \emph{kernel} of the operator) such that for all $f \in X$ the function $t \mapsto k_T(s, t) f(t)$ belongs to 
$L^1(\Omega, \Sigma, \mu)$ for almost all $s \in \Omega$ and the equality 
$$ (Tf)(s) = \int_{\Omega} k_T(s, t) f(t) d\mu(t) $$
holds for almost all $s \in \Omega$.

If $X$ is an ordered vector space, then the set $X_+ = \{ x \in X: x \geq 0\}$ is referred to a \emph{positive cone} of $X$.
An ordered vector space $X$ has a \emph{generating cone} if $X = X_+ - X_+$, i.e., every vector can be written as a difference of two positive vectors. 
The algebraic dual of a vector space $X$ is denoted by $X^\prime$.
If $T$ is an operator on a vector space $X$, then $T^\prime$ denotes the adjoint operator on $X^\prime$.
If $X$ is a topological vector space, then $X^*$ denotes the topological dual of $X$. 
If $T$ is an operator on a topological vector space $X$, then $T^*$ denotes the adjoint operator on $X^*$.
Let $y \in X$ and $\varphi \in X^\prime$, where $X$ is a vector space. The rank-one operator $T = y \otimes \varphi$ on $X$ is defined by 
$T x = \varphi(x) y$.

Let $X$ be a vector space and let $C$ be a cone in $X$. The cone $C$ is \emph{generated} by a set $S$ if every vector $x\in C$ can be written in the form 
$x=\lambda_1 s_1+\cdots+\lambda_n s_n$ for some $n\in \mathbb N$, nonnegative real numbers $\lambda_1,\ldots,\lambda_n$ and some vectors $s_1,\ldots,s_n \in S$. Clearly, we have $S\subseteq C$ and we write $C=\cone(S)$. If $S$ is finite, then $C$ is a \emph{finitely generated cone}. 
If $C$ is generated by a Hamel basis of $X$, it is called a \emph{Yudin cone}.  For the terminology
and details not explained here we refer the reader to \cite{LZ71}, \cite{AB06}, \cite{AA02} and \cite{AT}.

\section{Basic properties of the nonnegative rank of positive operators}\label{sec:nonnegative rank intro}

A positive operator $T\colon X\to Y$ between ordered vector spaces is said to have a \emph{finite nonnegative rank} if there exists a finite-dimensional Archimedean vector lattice $Z$ and positive operators $L\colon Z\to Y$ and $R\colon X\to Z$ such that $T= LR$. 

\[
\begin{tikzpicture}[>=Stealth, node distance=2.5cm, on grid]
  \node (X) at (0,0) {$X$};
  \node (Y) at (4,0) {$Y$};
  \node (Z) at (2,-2) {$Z$};

  \draw[->] (X) -- (Y) node[midway, above] {$T$};
  \draw[->] (X) -- (Z) node[midway, left, xshift=-4pt, yshift=-2pt] {$R$};
  \draw[->] (Z) -- (Y) node[midway, right, xshift=4pt, yshift=-2pt] {$L$};
\end{tikzpicture}
\]

If $T$ has a finite nonnegative rank, the minimal dimension $\dim Z$ through which $T$ admits a factorization with positive factors $R$ and $L$ is denoted by $\rank^+(T)$, and is called the \emph{nonnegative rank} of $T$. We define $\rank^+(T)=\infty$ if $T$ does not admit a factorization through any finite-dimensional Archimedean vector lattice via positive operators. Since every nonzero finite-dimensional Archimedean vector lattice is lattice isomorphic to $\mathbb R^k$ for some $k\in \mathbb N_0$ ordered coordinatewise, in the definition of a nonnegative rank we can replace $Z$ with $\mathbb R^{\dim Z}$. It is worthwhile noting that the definition of a nonnegative rank naturally extends the definition of a nonnegative rank of a nonnegative $m\times n$ matrix since a nonnegative $m\times n$ matrix can be considered as a positive operator from $\mathbb R^n$ to $\mathbb R^m$ both with standard coordinatewise ordering.   

A trivial yet important observation is that every positive operator with a finite nonnegative rank is a finite rank operator. Indeed, the range $\ran T$ is contained in the finite-dimensional vector space $\ran L$. Moreover, we have 
\begin{equation}\label{rank - nonnegative rank inequality}
\rank(T)\leq \rank^+(T).
\end{equation} 

\begin{lemma}\label{tensor_representation}
    Let $T\colon X\to Y$ be a nonzero positive operator between ordered vector spaces. 
    Then $T$ has a finite nonnegative rank if and only if there exists $k\in \mathbb N$, positive vectors $y_1,\ldots,y_k \in Y_+$ and positive linear functionals $\varphi_1,\ldots,\varphi_k\in X'$ such that 
    \begin{equation}\label{nonnegative_rank_tensors}
     T=\sum_{j=1}^k y_j\otimes \varphi_j.
    \end{equation}
    If $T$ has a finite nonnegative rank, then $\rank^+(T)$ equals to the minimal possible $k$ in \eqref{nonnegative_rank_tensors}.
\end{lemma}

\begin{proof}
    Suppose first that $T$ has a finite nonnegative rank $r:=\rank^+(T)$.  There exist positive operators $R\colon X\to \mathbb R^r$ and $L\colon \mathbb R^r\to Y$ such that $T=LR$. Let $\{e_1,\ldots,e_r\}$ be the set of standard basis vectors of $\mathbb R^r$. 
    For each $1\leq j\leq r$ denote  $y_j=Le_j$.   
    Pick any $x\in X$. Then there exist real numbers $\varphi_1(x),\ldots,\varphi_r(x)$ such that 
    $Rx=\varphi_1(x) e_1+\cdots+\varphi_r(x)e_r$. Since $\{e_1,\ldots,e_r\}$ is a basis for $\mathbb R^r$, each function $\varphi_j\colon X\to \mathbb R$ is a linear functional. Furthermore, positivity of $R$ yields that for each positive vector $x\in X_+$ and $1\leq j\leq r$ the scalars $\varphi_j(x)$ are nonnegative. Thus, $\varphi_1,\ldots,\varphi_r$ are positive linear functionals on $X$. Since for every $x\in X$ we have
    \[Tx=LRx=L\bigg(\sum_{j=1}^r\varphi_j(x)e_j\bigg)=\sum_{j=1}^r\varphi_j(x)Le_j,\]
    we conclude 
    \[T=\sum_{j=1}^rLe_j\otimes \varphi_j=\sum_{j=1}^ry_j\otimes \varphi_j.\]
    This shows that $\rank^+(T)$ is greater than or equal to the minimal $k$ from \eqref{nonnegative_rank_tensors}.    

    Suppose now that $T=\sum_{j=1}^ky_j\otimes \varphi_j$ for some $k\in \mathbb N$, positive vectors $y_1,\ldots,y_k\in Y_+$ and positive linear functionals $\varphi_1,\ldots,\varphi_k\in X'$. 
    We define the operator $R\colon X\to \mathbb R^k$ by 
    \[Rx:=(\varphi_1(x),\ldots,\varphi_k(x)).\]
    Since $\varphi_1,\ldots,\varphi_k$ are positive linear functionals, $R$ is a positive linear operator. 
    To define the operator $L\colon \mathbb R^k\to Y$, we first define $L_0$ on the standard basis $\{e_1,\ldots,e_k\}$ of $\mathbb R^k$ by $L_0e_j:=y_j$ for each $1\leq j\leq k$ and then we linearly extend $L_0$ to $L$ on $\mathbb R^k$.  Since the vector $y_j$ is positive for each $1\leq j\leq k$, the operator $L\colon \mathbb R^k\to Y$ is positive. 
    A similar calculation as in the proof of the opposite direction shows $T=LR$. This shows that $\rank^+(T)\leq k$ which finishes the proof.
\end{proof}

By \cite[Theorem 2.32]{AT}, each positive functional $\varphi\colon X\to \mathbb R$ on a completely metrizable ordered topological vector space $X$ with a closed and generating cone is automatically continuous. Since every linear operator $T\colon X\to Y$ between topological vector spaces with $X$ finite-dimensional and Hausdorff
is continuous, the following corollary immediately follows from Lemma \ref{tensor_representation}. 

\begin{corollary}\label{tensor_representation_continuous}
    Let $T\colon X\to Y$ be a nonzero positive operator between ordered topological vector spaces. If $X$ is completely metrizable with a closed and generating cone, then $T$ has a finite nonnegative rank if and only if there exists $k\in \mathbb N$, positive linear functionals $\varphi_1,\ldots,\varphi_k\in X^*$ and positive vectors $y_1,\ldots,y_k\in Y_+$ such that 
    \begin{equation}\label{nonnegative_rank_tensors_continuous}
     T=\sum_{j=1}^k y_j\otimes \varphi_j.
    \end{equation}
     If $T$ has a finite nonnegative rank, then $\rank^+(T)$ equals to the minimal possible $k$ in \eqref{nonnegative_rank_tensors_continuous}.
\end{corollary}

In the remaining part of this section we establish infinite-dimensional analogs of some results from \cite{CR93}.

\begin{proposition}\label{ideal property}
Suppose that
\[
\begin{tikzcd}
X \arrow[r, "A"] & Y \arrow[r, "T"] & Z \arrow[r, "B"] & W
\end{tikzcd}
\]
are positive operators between ordered vector spaces. If $T$ has a finite nonnegative rank, then the operator $BTA\colon X\to W$ also has finite nonnegative rank, and
\[
\rank^{+}(BTA)\leq \rank^{+}(T).
\]
\end{proposition}

\begin{proof}
Let us denote $r:=\rank^+(T)$. 
Since $T$ has a finite nonnegative rank, by Lemma \ref{tensor_representation} there exist positive vectors $z_1,\ldots,z_r \in Z$ 
and positive linear functionals $\varphi_1,\ldots,\varphi_r\in Y'$ such that 
$$T=\sum_{j=1}^r z_j\otimes \varphi_j.$$
Since  
$$BTA=\sum_{j=1}^r Bz_j\otimes A'\varphi_j,$$ Lemma \ref{tensor_representation} yields that $BTA$ also has a finite nonnegative rank which satisfies $\rank^+(BTA)\leq r$. 
\end{proof}

\begin{corollary}\label{ideal property min}
Suppose that
\[
\begin{tikzcd}
X \arrow[r, "A"] & Y \arrow[r, "B"] & Z  
\end{tikzcd}
\]
are positive operators between ordered vector spaces. If $A$ and $B$ have finite nonnegative ranks, then $BA\colon X\to Z$ has a finite nonnegative rank satisfying 
$$\rank^+(BA)\leq \min\{\rank^+(A),\rank^+(B)\}.$$
\end{corollary}

\begin{proof}
    Consider the diagram 
    \[
\begin{tikzcd}
X  \arrow [r,"I"] & X \arrow[r, "A"] & Y \arrow[r, "B"] & Z  
\end{tikzcd}
\]
of positive operators between ordered vector spaces. Since $A$ has a finite nonnegative rank, by Proposition \ref{ideal property} the operator $BA=BAI$ also has a finite nonnegative rank satisfying $\rank^+(BA) \leq \rank^+(A)$. Similarly we can prove $\rank^+(BA)\leq \rank^+(B)$.
\end{proof}

\begin{corollary}\label{nonnegative rank in finite dimensions}
    Let $T\colon X\to Y$ be a positive operator with a finite nonnegative rank between ordered vector spaces. If $\dim X$ or $\dim Y$ is finite, then 
    $$\rank^+(T)\leq \min\{\dim X,\dim Y\}.$$
\end{corollary}

 \begin{proof}
 Suppose that $X$ is finite-dimensional.
       Consider the diagram 
    \[
\begin{tikzcd}
X  \arrow [r,"I"] & X \arrow[r, "T"]  & Y
\end{tikzcd}
\] 
of positive operators.  By Corollary \ref{ideal property min} we have $$\rank^+(T)\leq \min\{\rank^+(T),\rank^+(I)\}\leq \rank^+(I)=\dim X.$$ 
Similarly we treat the case when $Y$ is finite-dimensional.
 \end{proof}

The proof of the following result once more relies on Lemma \ref{tensor_representation}, and is, therefore, omitted. 

\begin{proposition}
Let $A,B\colon X\to Y$ be positive operators between ordered vector spaces. If $A$ and $B$ have finite nonnegative ranks, then $A+B$ has a finite nonnegative rank satisfying
$$\rank^+(A+B)\leq \rank^+(A)+\rank^+(B).$$
\end{proposition}

We conclude this section with a result on the finite nonnegative rank of the adjoint of a positive operator between Banach lattices.

\begin{theorem}\label{thm:positive-rank-adjoint-invariance}
Let $T\colon X\to Y$ be a positive operator between Banach lattices. Then $T$ has a finite nonnegative rank if and only if $T^*$ has a finite nonnegative rank. In this case,
\[
\rank^{+}(T)=\rank^{+}(T^*).
\]
\end{theorem}

\begin{proof}
    Suppose that $T$ has a finite nonnegative rank $r$. Then there exist positive vectors $y_1,\ldots,y_r\in Y$ and positive linear functionals $\varphi_1,\ldots,\varphi_r\in X^*$
    such that $T=\sum_{j=1}^r y_j\otimes \varphi_j$. Let $\iota_{_X}$ and $\iota_{_Y}$ denote the canonical embeddings of $X$ and $Y$ into their biduals, respectively. It is well-known that $\iota_{_X}$ and $\iota_{_Y}$ are lattice isometries. In particular, they are positive operators. Then 
    $$T^*=\sum_{j=1}^r\varphi_j \otimes \iota_{_Y}(y_j)$$ shows that $T^*$ has a finite nonnegative rank satisfying $\rank^+(T^*)\leq \rank^+(T)$. Since the norm dual of a Banach lattice is again a Banach lattice, replacing $T$ by $T^*$ gives $\rank^+(T^{**})\leq \rank^+(T^*)\leq \rank^+(T)$. 

    Assume now that $T^{*}$ has a finite nonnegative rank. Then $T^{**}$ has a finite nonnegative rank as well. 
    Let $r=\rank^+(T^{**})$. 
    By Proposition \ref{ideal property} the operator $T^{**}\iota_{_X}\colon X\to Y^{**}$ has a finite nonnegative rank satisfying 
    $\rank^+(T^{**}\iota_{_X})\leq \rank^+(T^{**})$. Since $T^{**}\iota_{_X}=\iota_{_Y}T$, the operator $\iota_{_Y}T\colon X\to \iota_{_Y}(Y)$ has a finite nonnegative rank. Therefore, there exist positive vectors $y_1,\ldots,y_n\in Y$ and positive linear functionals $\varphi_1,\ldots,\varphi_n\in X^*$ for some $n\leq r$ such that 
    $$\iota_{_Y}T=\sum_{j=1}^n  \iota_{_Y}(y_j) \otimes \varphi_j .$$
    Hence, for each $x\in X$ we have 
    $$\iota_{_Y}(Tx)=\sum_{j=1}^n\varphi_j(x) \iota_{_Y}(y_j).$$
    Since $\iota_{_Y}$ is injective, for each $x\in X$ we have 
    $$Tx=\sum_{j=1}^n \varphi_j(x)y_j \qquad \textrm{or equivalently}\qquad T=\sum_{j=1}^ny_j\otimes \varphi_j.$$
    This finally shows that $T$ has a finite nonnegative rank satisfying $\rank^+(T)\leq r=\rank^+(T^{**})$ which finishes the proof.  
\end{proof}

\section{Finitely generated cones and the nonnegative rank}\label{sec:rank and nonnegative rank agree}

If $T\colon X\to Y$ is a positive operator between ordered vector spaces with $X_+$ finitely generated, then $TX_+$ is by \cite[Exercise 3.6.1]{AT} also finitely-generated. The following result shows that in general $TX_+$ is always contained in a finitely generated cone whenever $T$ has a finite non-negative rank.

 \begin{proposition}\label{poliederski_range_cone}
    Let $T\colon X\to Y$ be a positive finite-rank operator between ordered vector spaces. If $\rank^+(T)<\infty$, then the cone $TX_+$ is contained in a finitely generated cone. 
 \end{proposition}

 \begin{proof}
    Obviously we may assume that $T$ is nonzero. By assumption, there exist $k\in \mathbb N$, positive linear functionals $\varphi_1,\ldots,\varphi_k\in X'$ and positive vectors $y_1,\ldots,y_k\in Y$ such that 
    \[T=\sum_{j=1}^k y_j\otimes \varphi_j.\] 
    Since for each $x\in X_+$ we have
    \[Tx=\sum_{j=1}^k\varphi_j(x)y_j,\]
    $TX_+$ is contained in $\cone\{y_1,\ldots,y_k\}$.  
 \end{proof}

In Section 5 we are going to prove that there exists a rank three operator which does not have a finite nonnegative rank, yet the cone $TX_+$ is contained in a finitely generated cone. 

Let $T\colon X\to Y$ be a linear operator between ordered vector spaces. The \emph{range cone} of $T$ in $Y$ is the cone $C_T:=\ran T\cap Y_+$. Note that $C_T$ is indeed a cone as $C_T\cap (-C_T)\subseteq Y_+\cap (-Y_+)=\{0\}$ and $0\in C_T$ yield $C_T\cap (-C_T)=\{0\}$. If $T$ is a positive operator, then the cone $TX_+$ is always contained in the range cone $C_T$ of the operator $T$.

\begin{lemma}\label{generating range cone}
    Let $T\colon X\to Y$ be a positive operator between ordered vector spaces. If $X_+$ is generating in $X$, then the range cone $\ran T\cap Y_+$ is generating in $\ran T$. 
\end{lemma}

\begin{proof}
   Since the cone $X_+$ is generating in $X$, we have
    \[\ran T = TX=TX_+-TX_+ \subseteq C_T-C_T.\]
    Since we always have $C_T-C_T\subseteq \ran T$, we have $\ran T=C_T-C_T$, so that $C_T$ is generating in $\ran T$. 
\end{proof}

The following theorem is an infinite-dimensional version of \cite[Theorem 4.1]{CR93}. It provides sufficient conditions for a positive operator to have a finite nonnegative rank.

\begin{theorem}\label{yudin rank+=rank}
     Let $T\colon X\to Y$ be a positive finite-rank operator between ordered vector spaces. Suppose that the cone $X_+$ is generating in $X$. Then $\rank(T)=\rank^+(T)$ in either of the following cases. 
     \begin{enumerate}
         \item[(i)] The range cone of $T$ is a Yudin cone in $\ran T$;
         \item[(ii)] $\rank(T)\leq 1$;
         \item[(iii)] $\rank(T)=2$  and $Y$ is a topological ordered vector space with a closed cone $Y_+$.
        \item[(iv)]  $\rank^+(T)<\infty$, $\min\{\dim X,\dim Y\}\leq 3$ and $Y$ is a topological ordered vector space with a closed cone $Y_+$.
     \end{enumerate} 
\end{theorem}

\begin{proof}
(i) Since the range cone $C_T:=\ran T\cap Y_+$ is a Yudin cone in a finite-dimensional space $\ran T$, it is generated by a positive Hamel basis $\mathcal B:=\{y_1,\ldots,y_n\}$ of $\ran T$. By \cite[Theorem 3.21]{AT} (see also \cite{Yud39}) we conclude that $\ran T$ is a finite-dimensional Archimedean vector lattice. 

To prove that $T$ has a finite nonnegative rank, we construct positive operators $L\colon \ran T\to Y$ and $R\colon X\to \ran T$ such that $T=LR$. For the operator $L\colon \ran T\to Y$ we take the inclusion operator, which is obviously positive. To define the operator  $R\colon X\to \ran T$, we first introduce the dual basis 
$\{\varphi_1,\ldots,\varphi_n\}$ of $\mathcal B$ and set
$$R=\sum_{k=1}^n y_k \otimes T^\prime \varphi_k.$$ 
Since each vector $y\in C_T$ is a nonnegative linear combination of vectors from $\mathcal B$, each linear functional $\varphi_j$ is positive, which yields positivity of $T^\prime \varphi_j$. Since $R$ is a finite sum of positive rank-one operators, it is a positive operator. 

We claim that $T=LR$. To this end, pick any vector $x\in X_+$. Then $Tx\in C_T$, so that there exist nonnegative real numbers $\lambda_1,\ldots,\lambda_n$ such that $Tx=\lambda_1y_1+\cdots+\lambda_ny_n$. Since  $\varphi_i(y_j)=1$ if $i=j$ and zero otherwise, for any $1\leq k\leq n$ we have $\varphi_k(Tx)=\sum_{j=1}^n \lambda_j \varphi_k(y_j)=\lambda_k$ yielding 
\begin{align*}
LRx &= \sum_{k=1}^n(y_k\otimes T^\prime \varphi_k)x=\sum_{k=1}^n\varphi_k(Tx)y_k=Tx.    
\end{align*}
Since $T$ and $LR$ agree on the generating cone $X_+$ of $X$, we have $T=LR$. This proves that $\rank^+(T)\leq \rank(T)$. Since the reverse inequality always hold, we have $\rank^+(T)=\rank(T)$. 

(ii) and (iii) If $T=0$, then $T$ clearly factors through the zero-dimensional space. So, suppose $T\neq 0$. In both cases when $\rank T$ is either $1$ or $2$, we will prove that the range cone is a Yudin cone, so that an application of (i) will conclude the proof. 

Consider now the case $\rank T=1$. Since $X_+$ is generating in $X$, there exists a vector $x\in X_+$ such that $y:=Tx$ is nonzero and positive. Since $\rank T=1$, the range cone $C_T$ is generated by the Hamel basis $\{y\}$ of $\ran T$. Now we apply (i).

Suppose now that $\ran T=2$. Since $Y_+$ is closed in $Y$, by \cite[Theorem~2.3]{AT} the topological vector space $Y$ is Hausdorff. Since $\ran T$ is finite-dimensional in $Y$, it is closed. Hence, the range cone $C_T=\ran T\cap Y_+$ is also closed in $Y$. Since the range cone of $T$ is generating, \cite[Theorem~3.5]{ABPW} yields that $C_T$ is a Yudin cone in $\ran T$. 

(iv) If $\min\{\dim X,\dim Y\}\leq 2$, then the rank of $T$ is at most two. So, by (ii) or (iii) we have $\rank(T)=\rank^+(T)$.

Assume now that $\min\{\dim X,\dim Y\}=3$.  Then $\rank(T)\leq 3$. As above, if $\rank(T)\leq 2$, then $\rank(T)=\rank^+(T)$. Hence, we may assume $\rank(T)=3$. To finish the proof observe that Corollary \ref{nonnegative rank in finite dimensions} yields $3=\rank(T)\leq \rank^+(T)\leq 3$. 
\end{proof}
 
The following corollary immediately follows from Theorem \ref{yudin rank+=rank}.

\begin{corollary}\label{rank=rank+}
   Let $T\colon X\to Y$ be a positive operator between normed lattices. 
   If $T$ has a finite nonnegative rank, then 
   $\rank^+(T)=\rank(T)$ in either of the following cases. 
   \begin{enumerate}
       \item[(i)] $\rank(T)\leq 2$;
       \item[(ii)] $\min\{\dim X,\dim Y\}\leq 3$.
   \end{enumerate}
\end{corollary}

\section{An example of a positive rank three operator with infinite nonnegative rank}\label{sec:example}

It follows from \cite[Proposition 3.1]{BL09} that for any $k \geq 3$ there exist $n \in \mathbb{N}$ and an $n\times n$ nonnegative matrix $T$ such that $\rank^+(T)=k$ and $\rank(T)=3$. This poses the following question.

\begin{question} 
Does there exist a rank-three positive operator $T\colon X\to Y$ between Banach lattices such that $\rank^{+}(T)=\infty$?
\end{question}

In this section we will provide an example of such an operator. Let $X$ be either $C[0,2\pi]$ or $L^p[0,2\pi]$ for some $1\leq p\leq \infty$ and consider the integral operator $\mathscr S\colon X\to X$ defined as 
$$(\mathscr S f)(s)=\int_0^{2\pi} f(t)(1+\cos(s-t))\,dt$$
where in the case $1\leq p\leq \infty$ the equality above is meant to hold for almost every $s\in[0,2\pi]$. 
Let $u_1,u_2,u_3\in X$ be given by $u_1(t)=\cos t$, $u_2(t)=\sin t$, and $u_3(t)=1$. 
Define linear functionals $\varphi_1,\varphi_2,\varphi_3\colon X\to\mathbb R$ by
\[
\varphi_1(f)=\int_{0}^{2\pi} f(t)\cos t\,dt,\qquad
\varphi_2(f)=\int_{0}^{2\pi} f(t)\sin t\,dt,\qquad
\varphi_3(f)=\int_{0}^{2\pi} f(t)\,dt.
\]
Since 
$$(\mathscr S f)(s)= \int_0^{2\pi} f(t) \cos t\,dt \cdot \cos s+\int_0^{2\pi} f(t) \sin t\,dt \cdot \sin s + \int_0^{2\pi} f(t) \,dt,$$
the operator $\mathscr S$ admits the representation
\begin{equation}\label{tensor representation od S}
\mathscr S = u_1\otimes\varphi_1 + u_2\otimes\varphi_2 + u_3\otimes\varphi_3.
\end{equation}

We will show in Theorem \ref{rank three infinite nonnegative rank} that the operator $\mathscr S$ has infinite nonnegative rank. To this end, we first establish the following two lemmas.

\begin{lemma}\label{lastnosti S}
The operator $\mathscr S\colon X\to X$ is a positive rank-three operator with $\ran \mathscr S=\Span\{u_1,u_2,u_3\}$.
\end{lemma}

\begin{proof}
Since the function $(s,t)\mapsto 1+\cos(s-t)$ is nonnegative, the operator $\mathscr S$ is positive. The representation \eqref{tensor representation od S} shows that 
the range $\ran \mathscr S $ is contained in $\Span\{u_1,u_2,u_3\}$. From $\mathscr S u_1=\pi u_1$, $\mathscr S u_2=\pi u_2$ and $\mathscr S u_3=2 \pi u_3$ we conclude $\rank(\mathscr S )=3$.
\end{proof}

\begin{lemma}\label{no_finite_sum_rep}
Let $X$ and $Y$ each be either $C[0,2\pi]$ or $L^{p}[0,2\pi]$ for some $1 \le p \le \infty$.
Then there does not exist a positive integer $N \in \mathbb{N}$ and collections of nonnegative functions
$a_1,\dots,a_N \in X$ and $b_1,\dots,b_N \in Y$
such that
$$
1+\cos(s-t)=\sum_{k=1}^N a_k(s)\, b_k(t)
$$
holds almost everywhere on $[0,2\pi]\times[0,2\pi]$.
\end{lemma}

\begin{proof}
Assume that there exist $N\in \mathbb N$, and nonnegative functions $a_1,\ldots,a_N\in X$ and $b_1,\ldots,b_N\in Y$ such that 
$$
1+\cos(s-t)=\sum_{k=1}^N a_k(s)\, b_k(t)
$$
holds almost everywhere on $[0,2\pi]^2$. We define 
$$G(s,t)=1+\cos(s-t)-\sum_{k=1}^N a_k(s)\, b_k(t).$$
Then $G(s,t)=0$ holds almost everywhere on $[0,2\pi]^2$.

We claim that there exists a measurable set $T_0\subseteq [0,2\pi]$ with full measure such that
\begin{enumerate}
    \item[(i)] functions $a_k$ and $b_k$ are defined on $T_0$ for all $1\leq k\leq N$;
    \item[(ii)] for every $t\in T_0$ we have $G(s,t)=0$ for almost every $s\in [0,2\pi]$.
\end{enumerate}
To this end, let us consider the set 
$$E=\{(s,t)\in [0,2\pi]^2:\; G(s,t)\neq 0\}.$$
By assumption, the two-dimensional Lebesgue measure of $E$ is zero. By the Tonelli theorem for the indicator function $\chi_E$ of the set $E$ we have 
$$0= \int_{0}^{2\pi}\int_{0}^{2\pi} \chi_{E}(s,t)\,ds\,dt
= \int_{0}^{2\pi}\left(\int_{0}^{2\pi} \chi_{E}(s,t)\,ds\right) dt.$$
Since the function 
$$t\mapsto \int_0^{2\pi} \chi_E(s,t)\,ds$$ is nonnegative, it follows that 
$$\int_{0}^{2\pi}\chi_{E}(s,t)\,ds = 0$$
for almost every $t$. 

Next, both functions $a_k$ and $b_k$ are defined at least almost everywhere, since they belong to $X$ and $Y$, respectively.  For each $k$, let
$$B_k=\{t\in[0,2\pi]: a_k(t) \textrm{ or } b_k(t) \textrm{ is not defined}\}.$$
Then each $B_k$ has measure zero, so the finite union 
$B:= \bigcup_{k=1}^{N} B_k$
also has measure zero.

Finally, define
$$T_0 := \bigg\{ t\in[0,2\pi]\;: 
\int_{0}^{2\pi}\chi_{E}(s,t)\,ds = 0 \bigg\} \setminus B.
$$
and observe that $T_0$ has full measure. In particular, $T_0$ is uncountable. 

For each $t\in T_0$ we consider the function $F_t$ defined as
$$F_t(s)=1+\cos(s-t)=\sum_{k=1}^Na_k(s)b_k(t).$$
The zero set $Z_t$ of the function $F_t$ consists of all $s\in [0,2\pi]$ such that $s-t=\pi$. Hence, the zero set $Z_t$ consists of a single point. In particular, the family $\{Z_t:\; t\in T_0\}$ contains uncountably many distinct singleton sets. 

On the other hand, for every $t\in T_0$ the identity
$$F_t(s)=\sum_{k=1}^Na_k(s)b_k(t)$$
holds for almost every $s\in [0,2\pi]$. For each $t\in T_0$ define the index set 
$$I(t):=\bigg\{k\in \{1,\ldots,N\}:\; b_k(t)>0\bigg\}.$$ Since functions $a_1,\ldots,a_N,b_1,\ldots,b_N$ are nonnegative, for a fixed $t\in T_0$ we have $F_t(s)=0$ if and only if $a_k(s)=0$ for every $k\in I(t)$. Hence, 
$$Z_t=\bigcap_{k\in I(t)}Z(a_k),$$
where $Z(f)$ denotes the zero set of a real function $f$ on $[0, 2\pi]$.
Consequently, for every $t\in T_0$ the zero set $Z_t$ is of the form 
$$\bigcap_{k\in J}Z(a_k)$$
for some subset $J\subseteq \{1,\ldots,N\}$.
Since there are only $2^N$ possible subsets $J$, there are at most $2^N$ such intersections. Therefore, the family $\{Z_t:\; t\in T_0\}$ can have cardinality at most $2^N$, contradicting the fact that it contains uncountably many distinct zero sets. 
\end{proof}

\begin{theorem}\label{rank three infinite nonnegative rank}
Let $X$ be either $C[0,2\pi]$ or $L^p[0,2\pi]$ for some $1\leq p\leq \infty$.
The nonnegative rank of the operator $\mathscr S\colon X\to X$ is infinite.
\end{theorem}

\begin{proof}
Suppose, for the sake of contradiction, that the nonnegative rank of $\mathscr S$ is finite. 
By Corollary \ref{tensor_representation_continuous}, $\mathscr S$ can be written as a finite sum
\[
\mathscr S =\sum_{k=1}^N a_k\otimes\varphi_k,
\]
where for each $1\leq k\leq N$ the vector $a_k\in X$ satisfies $a_k\ge0$ in $X$ and $\varphi_k\in X^*$ is a positive bounded linear functionals on $X$. Without loss of generality we may assume that none of the functions $a_k$ is zero almost everywhere for each $1\leq k\leq N$.
We will prove that there exist nonnegative functions $b_1,\ldots,b_N$ all belonging to the same space $L^q[0,2\pi]$ for some $1\leq q\leq \infty$ (depending on $X$) such that
$$1+\cos(s-t)=\sum_{k=1}^Na_k(s)b_k(t)$$
holds for almost every $(s,t)\in [0,2\pi]^2$. This, however will be in contradiction with Lemma \ref{no_finite_sum_rep} and the proof will be finished. 

We first consider the case $X=C[0,2\pi]$. By the Riesz representation theorem, each $\varphi_k$ has the form
\[
\varphi_k(f)=\int_{0}^{2\pi}f(t)\,d\mu_k(t)
\]
for some finite positive Radon measure $\mu_k$ on $[0,2\pi]$.
Hence,
\[
\int_{0}^{2\pi}(1+\cos(s-t))f(t)\,dt=(\mathscr S f)(s)=\sum_{k=1}^N a_k(s)\int_0^{2\pi} f(t)\,d\mu_k(t).
\]
Fix $s\in [0,2\pi]$. By the uniqueness part of the Riesz representation theorem, for each $s\in [0,2\pi]$ we have
\begin{equation}\label{eq:measure_identity}
(1+\cos(s-t))\,dt
   = \sum_{k=1}^N a_k(s)\,d\mu_k(t).
\end{equation}

By the Lebesgue–Radon–Nikodým decomposition, for each $1\leq k\leq N$ there exist a nonnegative function
$b_k\in L^{1}[0,2\pi]$ and a positive measure $\lambda_k$ singular with respect to Lebesgue measure such that
\[
d\mu_k(t)=b_k(t)\,dt + d\lambda_k(t).
\]
Substituting into \eqref{eq:measure_identity} gives
\[
(1+\cos(s-t))\,dt
= \sum_{k=1}^N a_k(s)b_k(t)\,dt
  + \sum_{k=1}^N a_k(s)\,d\lambda_k(t).
\]
Since for each $s\in [0,2\pi]$ the Lebesgue measure and the measure
\(
\sum_{k=1}^N a_k(s)\lambda_k
\)
are mutually singular, for each $s\in [0,2\pi]$ it follows that
\[
\sum_{k=1}^N a_k(s)\lambda_k = 0.
\]
Therefore, \eqref{eq:measure_identity} yields that for every $s\in [0,2\pi]$ we have
\[
1+\cos(s-t)=\sum_{k=1}^N a_k(s)b_k(t)
\]
for almost every $t\in [0,2\pi]$. 

Now let us consider the case $X=L^p[0,2\pi]$ for some $1\leq p<\infty$. By the representation theorem of bounded linear functionals on $L^p$-spaces over $\sigma$-finite measure spaces (see \cite[Theorem 6.16]{Rudin}), for each $1\leq k\leq N$ there exists a function $b_k\in L^q[0,2\pi]$ where $q$ is the conjugated exponent of $p$ such that 
$$\varphi_k(f)=\int_0^{2\pi} f(t)b_k(t)\,dt$$ 
holds for each $f\in L^p[0,2\pi]$. Since $\varphi_k$ is positive, function $b_k$ is nonnegative almost everywhere. This yields that for each $f\in L^p[0,2\pi]$ we have
\begin{align*}
\int_0^{2\pi}f(t)(1+\cos(s-t))\,dt&=(\mathscr Sf)(s)=\sum_{k=1}^Na_k(s)\int_0^{2\pi}f(t)b_k(t)\,dt 
\end{align*}
for almost every $s\in [0,2\pi]$. 
Therefore, the operator $\mathscr S\colon L^p[0,2\pi]\to L^p[0,2\pi]$ is an integral operator, represented with integral kernels $(s,t)\mapsto 1+\cos(s-t)$ and $(s,t)\mapsto \sum_{k=1}^n a_k(s)b_k(t)$. Since $L^p[0,2\pi]$ is order dense in $L^0[0,2\pi]$, by \cite[Theorem 5.5]{AA02} we have  
$$1+\cos(s-t)=\sum_{k=1}^Na_k(s)b_k(t)$$ 
for almost every $(s,t)\in [0,2\pi]^2$. 

Lastly, let us consider the case $X=L^\infty[0,2\pi]$. By the representation theorem of the dual space of $L^\infty[0,2\pi]$ (see  e.g. \cite[Theorem 2.3]{YH52}), for each $1\leq k\leq N$ there exists a finitely additive positive Borel measure $\mu_k$ on $[0,2\pi]$ which is absolutely continuous with respect to the Lebesgue measure such that for each $f\in L^\infty[0,2\pi]$ we have
$$\varphi_k(f)=\int_0^{2\pi}f(t)\,d\mu_k(t).$$ 
We claim that $\mu_k$ is countably additive. To this end, observe first that the operator $\mathscr S$ is by \cite[Corollary 5.4]{AA02} order continuous. Since we have $0\leq a_k\otimes \varphi_k\leq \mathscr S$ and $a_k$ is not zero almost everywhere, the positive linear functional $\varphi_k$ is also order continuous. If $(E_n)_{n\in\mathbb N}$ is an increasing sequence of measurable subsets of $[0,2\pi]$, then $\chi_{E_n}\uparrow \chi_E$ where $E=\cup_{n=1}^\infty E_n$. Since $\varphi_k$ is order continuous, we have $\varphi_k(\chi_{E_n})\uparrow \varphi_k(\chi_E)$ yielding
$\mu_k(E_n)\uparrow \mu_k(E)$. This shows countable additivity of $\mu_k$.  

Once again, by the Radon–Nikodým theorem, there exists a nonnegative function $b_k\in L^1[0,2\pi]$ such that $d\lambda_k=b_k\,dt$. As in the previous case, this yields that for each $f\in L^\infty[0,2\pi]$ we have
\begin{align*}
\int_0^{2\pi}f(t)(1+\cos(s-t))\,dt&=(\mathscr Sf)(s)=\sum_{k=1}^Na_k(s)\int_0^{2\pi}f(t)b_k(t)\,dt 
\end{align*}
for almost every $s\in [0,2\pi]$, and so by a similar reasoning as above we have 
$$1+\cos(s-t)=\sum_{k=1}^Na_k(s)b_k(t)$$ 
for almost every $(s,t)\in [0,2\pi]^2$. 

To finish the proof, we apply Lemma \ref{no_finite_sum_rep} in all three cases.
\end{proof}

\section{On the cones associated with the operator $\mathscr S$}\label{sec:S-range cone}

Recall that the operator $\mathscr S\colon C[0,2\pi]\to C[0,2\pi]$ is defined by
\[
(\mathscr Sf)(x)=\int_{0}^{2\pi} f(t)\bigl(1+\cos(x-t)\bigr)\,dt,
\]
for $f\in C[0,2\pi]$ and $x\in[0,2\pi]$. 
By Theorem \ref{rank three infinite nonnegative rank} we know that $\mathscr S$ is a positive operator of rank three with infinite nonnegative rank. In this section we study $\mathscr S$ in more detail. In particular, we determine its range cone $C_{\mathscr S}$ and the cone $\mathscr S\bigl(C[0,2\pi]_+\bigr)$.

Recall also the definitions of bounded linear functionals $\varphi_1,\varphi_2,\varphi_3\colon C[0,2\pi]\to\mathbb R$:
\[
\varphi_1(f)=\int_{0}^{2\pi} f(t)\cos t\,dt,\qquad
\varphi_2(f)=\int_{0}^{2\pi} f(t)\sin t\,dt,\qquad
\varphi_3(f)=\int_{0}^{2\pi} f(t)\,dt.
\]
Next, let $u_1,u_2,u_3\in C[0,2\pi]$ be given by $u_1(t)=\cos t$, $u_2(t)=\sin t$, and $u_3(t)=1$. Then $\mathscr S$ admits the representation
\[
\mathscr S = u_1\otimes\varphi_1 + u_2\otimes\varphi_2 + u_3\otimes\varphi_3.
\]
 
\begin{proposition}\label{range cone of F}
For a function $f\in C[0,2\pi]$ we have $\mathscr Sf\geq 0$ if and only if  
\begin{equation}\label{Tf>=0} 
\varphi_3(f) \geq \sqrt{\varphi_1(f)^2+\varphi_2(f)^2}.
\end{equation}
\end{proposition}

\begin{proof}
    An easy calculus exercise shows that the minimum of the function $x\mapsto a \cos x+b\sin x + c$ on the interval $[0,2\pi]$ is $c-\sqrt{a^2+b^2}$. Since
    \[(\mathscr Sf)(x)= \varphi_1(f) \cos x + \varphi_2(f)\sin x + \varphi_3(f),\] 
    the function $\mathscr Sf$ is nonnegative on $[0,2\pi]$ if and only if 
     \[\varphi_3(f)-\sqrt{\varphi_1(f)^2+\varphi_2(f)^2} \geq 0\] 
     which proves \eqref{Tf>=0}.
\end{proof}

We start by determining the range cone $C_{\mathscr S}$ of $\mathscr S$.
Let 
$$C=\{(a,b,c)\in \mathbb R^3:\; c\geq \sqrt{a^2+b^2}\}$$ 
be the standard ice-cream cone in $\mathbb R^3$.

\begin{proposition}\label{cone isomorphism}
The mapping $\Phi\colon (\mathbb R^3, C)\to (\ran \mathscr S,C_{\mathscr S})$ given by 
$$\Phi(a,b,c)=au_1+bu_2+cu_3$$ is an order and topological isomorphism. 
\end{proposition}

\begin{proof}
    Since $\ran \mathscr S=\Span\{u_1,u_2,u_3\}$,  $\Phi$ is a well-defined linear surjection. Therefore,  it is an isomorphism.  Moreover, since $\Phi$ maps between finite-dimensional Hausdorff spaces, it is a topological isomorphism. 

    To prove that $\Phi$ is an order isomorphism, note that $(a,b,c)\in C$ if and only if $c\geq \sqrt{a^2+b^2}$ that is in fact equivalent to $au_1+bu_2+cu_3\in C[0,2\pi]_+$. 
\end{proof}
 
Next we will provide an explicit description of the cone $\mathscr S\big(C[0,2\pi]_+\big)$. It turns out that its description is significantly harder to determine compared to that one of the range cone $C_{\mathscr S}$ of $\mathscr S$. We will need the \emph{Poisson kernel} $P_r$ ($0 < r < 1$) given by
\[P_r(t)=\frac{1-r^2}{1-2r\cos t+r^2}\]
for every $t\in [0,2\pi]$. It is well-known (see e.g. \cite{Rudin}) that $P_r$ is a nonnegative continuous function satisfying 
\begin{equation}\label{poisson_vrsta}
\int_0^{2\pi}P_r(t)\,dt=2\pi \quad \textrm{and}\quad P_r(t)=\sum_{n=-\infty}^\infty r^{|n|}e^{int}=1+2\sum_{n=1}^\infty r^n\cos nt.
\end{equation}
 Since $0<r<1$, the trigonometric series above converges uniformly to $P_r$ on $[0,2\pi]$. Therefore, the standard exchange of summation and integration gives us
 \begin{equation}\label{poissonove formule}
 \int_0^{2\pi}P_r(t)\cos t\,dt = 2\pi r \qquad \textrm{and}\qquad \int_0^{2\pi}P_r(t)\sin t\,dt = 0. 
 \end{equation}

\begin{theorem}\label{image of cone}
We have
 \begin{equation}\label{Tf>0} 
 \Phi^{-1}\big(\mathscr S(C[0,2\pi]_+)\big)=\big\{(a,b,c)\in \mathbb R^3:\; c> \sqrt{a^2+b^2}\big\}\cup\{(0,0,0)\}.
 \end{equation}
\end{theorem}

\begin{proof}
    Since the operator $\mathscr S$ is positive by Lemma \ref{lastnosti S},  $\mathscr S(C[0,2\pi]_+)$ is contained in the range cone of $\mathscr S$. By Proposition \ref{range cone of F}, each function $f\in C[0,2\pi]_+$ satisfies
    \eqref{Tf>=0}. Suppose there were some nonzero nonnegative function $f\in C[0,2\pi]_+$ satisfying 
    \begin{equation}\label{enacaj ne velja}
        \varphi_1(f)^2+\varphi_2(f)^2 = \varphi_3(f)^2.
    \end{equation} 
    The identity
    \[\varphi_1(f)+i\varphi_2(f)=\int_0^{2\pi}f(t)e^{it}\,dt\]
    immediately yields
    \[
    \varphi_3(f)^2 = \varphi_1(f)^2+\varphi_2(f)^2=\left|\int_0^{2\pi}f(t)e^{it}\,dt\right|^2\leq \]
    \[ \le \left(\int_0^{2\pi}|f(t)e^{it}|\,dt\right)^2=\left(\int_0^{2\pi} f(t) \,dt\right)^2 = \varphi_3(f)^2, \]
    from where we conclude
    \[\left|\int_0^{2\pi}f(t)e^{it}\,dt\right|=\int_0^{2\pi}|f(t)e^{it}|\,dt=\int_0^{2\pi} f(t) \,dt.\]
    An application of \cite[Theorem 1.39]{Rudin} provides a constant $\alpha$ such that for almost every $t\in [0,2\pi]$ we have
    $\alpha f(t)e^{it}=f(t)$. Since $f$ is continuous, the latter equality holds for each $t\in [0,2\pi]$. This would imply that $t\mapsto e^{it}$ is constant on some
    nonempty open set as $f$ is nonzero, which is impossible. This contradiction shows that in \eqref{Tf>=0} we have strict inequality which proves the containment ``$\subseteq$''.

    To prove the containment ``$\supseteq$'' take any triplet $(a,b,c)\in \mathbb R^3$ with $c>\sqrt{a^2+b^2}$. We need to find a function $f\in C[0,2\pi]_+$ such that $\Phi(a,b,c)=\mathscr Sf$, i.e., 
     \[a=\int_0^{2\pi}f(t)\cos t\,dt, \quad b=\int_0^{2\pi}f(t)\sin t\,dt \quad \textrm{and} \quad c=\int_0^{2\pi}f(t)\,dt.\]
     To this end, let us write $a+ib=Re^{i\theta}$ with $R=\sqrt{a^2+b^2}$ and choose $r\in (\tfrac Rc,1)$.
     Pick $\alpha$ such that $\cos \alpha=\frac{R}{cr}$ and consider the positive measure 
     \[\mu:=\frac{c}{2}\gamma_{\theta-\alpha}+\frac{c}{2}\gamma_{\theta+\alpha} , \]
     where $\gamma_s$ is meant to be the Dirac measure $\delta_{s'}$ where $s'$ is the unique point in $[0,2\pi]$ such that $s'$ and $s$ differ by an integer multiple of $2\pi$. 
     
     We claim that the function $f\colon [0,2\pi]\to \mathbb R$ defined as 
     \[f(t)=\frac{1}{2\pi}\int_0^{2\pi}P_r(t-s)\,d\mu(s)\]
     is the required function. First, observe that $f$ is nonnegative since $P_r$ is nonnegative and $\mu$ is a positive measure. To prove that $f$ is continuous, we can write
     \begin{align*}
         |f(t)-f(t_0)| &= \frac{1}{2\pi}\big|\int_0^{2\pi}(P_r(t-s)-P_r(t_0-s))\,d\mu(s)\big|\\
         & \leq \frac{1}{2\pi}\int_0^{2\pi}|(P_r(t-s)-P_r(t_0-s))|\,d\mu(s) \\
         &\leq \frac{1}{2\pi}\sup_{0\leq s\leq 2\pi}|(P_r(t-s)-P_r(t_0-s))| \cdot \mu([0,2\pi])
     \end{align*}
     and apply uniform continuity of $P_r$ to see that $f(t)\to f(t_0)$ as $t \to t_0$. 

     To finish the proof we need to calculate the integrals. First, by using the Fubini theorem we have
     \begin{align*}
     \int_0^{2\pi}f(t)\,dt & =\frac{1}{2\pi}\int_0^{2\pi}\bigg(\int_0^{2\pi} P_r(t-s)\,d\mu(s)\bigg)\,dt\\
     &=\frac{1}{2\pi}\int_0^{2\pi}\bigg(\int_0^{2\pi} P_r(t-s)\,dt\bigg)\,d\mu(s)=\mu([0,2\pi])=c.
     \end{align*}
     Similarly, but with more work we obtain 
     \begin{align*}
         \int_{0}^{2\pi} f(t)\cos t\,dt &= \frac{1}{2 \pi}\int_0^{2\pi}\bigg(\int_0^{2\pi} P_r(t-s) \cos t\,d\mu(s)\bigg)\,dt\\
         &=\frac{1}{2 \pi}\int_0^{2\pi}\bigg(\int_0^{2\pi} P_r(t-s)\cos t\,dt\bigg)\,d\mu(s)\\
         &=\frac{1}{2 \pi }\int_0^{2\pi}\bigg(\int_0^{2\pi} P_r(u)\cos(u+s)\,du\bigg)\,d\mu(s).
     \end{align*}
     Now an application of the cosine addition formula and \eqref{poissonove formule} yield 
    \begin{align*}
    \int_0^{2\pi}f(t)\cos t\,dt &= r\int_0^{2\pi}\cos s\,d\mu(s)=\frac{rc}{2}\big(\cos(\theta-\alpha)+\cos(\theta+\alpha)\big) \\
    &=rc\cos\theta\cos\alpha=R\cos\theta=a.
    \end{align*}
    The proof of the equality $b=\int_0^{2\pi}f(t)\sin t\,dt$ is very similar, so we omit it. 
\end{proof}

Proposition \ref{cone isomorphism} and Theorem \ref{image of cone} immediately yield the following corollary. 

\begin{corollary}
   The cone $\mathscr S (C[0,2\pi]_+)$ is dense in the range cone $C_{\mathscr S}$ of $\mathscr S$, and it is not 
   finitely generated in $C[0,2\pi]$. 
\end{corollary}

\subsection*{Acknowledgements}
The first author was supported by the Slovenian Research and Innovation Agency program P1-0222.
The second author was supported by the Slovenian Research and Innovation Agency program P1-0222 and grant J1-50002.

\subsection*{Data Availability Statement}

No datasets were generated or analysed during the current study.

\subsection*{Conflicts of Interest}

The authors declare that they have no conflicts of interest relevant to the content of this article.

\end{document}